\documentclass[12pt,reqno]{amsart}
\usepackage{amssymb,amsmath,amsthm,hyperref}
\newcommand{\sg}{\textnormal{sg}}

\newtheorem{theorem}{Theorem}

\newtheorem{corollary}[theorem]{Corollary}
\newtheorem{proposition}[theorem]{Proposition}

\theoremstyle{remark}
\newtheorem*{remark}{Remark}
\theoremstyle{definition}

\newtheorem{definition}[theorem]{Definition}

\numberwithin{theorem}{section} \numberwithin{equation}{section}
\numberwithin{example}{section}

\title[simultaneous representations of primes by binary quadratic forms]{Threefield identities and simultaneous representations of primes by binary quadratic forms }

\author{Eric Mortenson}

\begin{document}

\date{7 August  2012}

\begin{abstract}{Kaplansky [$2003$] proved a theorem on the simultaneous representation of a prime $p$ by two different principal binary quadratic forms.   Later, Brink found five more like theorems and claimed that there were no others.  By putting Kaplansky-like theorems into the context of threefield identities after Andrews, Dyson, and Hickerson, we find that there are at least two similar results not on Brink's list.  We also show how such theorems are related to results of Muskat on binary quadratic forms.}
\end{abstract}

\address{Department of Mathematics, The University of Queensland,
Brisbane, Australia 4072}

\email{etmortenson@gmail.com}

\maketitle

\section{Notation}
 Let $q$ be a complex number with $0<|q|<1$.  We recall some basics:
\begin{gather}
 (x)_{\infty}=(x;q)_{\infty}:=\prod_{i\ge 0}(1-q^ix),\\
{\text{and }}\ \ j(x;q):=(x)_{\infty}(q/x)_{\infty}(q)_{\infty}=\sum_{n}(-1)^nq^{\binom{n}{2}}x^n,\label{equation:JTPid}
\end{gather}
where in the last line the equivalence of product and sum follows from Jacobi's triple product identity.  We keep in mind the easily deduced fact that $j(q^n,q)=0$ for $n\in \mathbb{Z}.$  The following are special cases of the above definition.  Let $a$ and $m$ are integers with $m$ positive.  Define
\begin{gather}
J_{a,m}:=j(q^a;q^m), \ \ \overline{J}_{a,m}:=j(-q^a;q^m), \ {\text{and }}J_m:=J_{m,3m}=\prod_{i\ge 1}(1-q^{mi}).
\end{gather}

\section{Introduction}

Let $\Delta$ be a negative integer with $\Delta\equiv 0 \pmod 4$ (resp. $\Delta\equiv 1 \pmod 4$ ).  Recall that the principal binary quadratic form $F(x,y)$ of discriminant $\Delta$ is defined to be $x^2-\frac{\Delta}{4}y^2$ (resp. $x^2+xy+\frac{1-\Delta}{4}y^2$).  Kaplansky \cite{K} proved the following theorem on the simultaneous representation of a prime $p$ by two different principal binary quadratic forms:

\begin{theorem}\cite{K}\label{theorem:Kap-Thm}  A prime $p$, where $p\equiv 1 \pmod {16}$, is representable by both or none of the quadratic forms $x^2+32y^2$ and $x^2+64y^2.$  A prime $p$, where $p\equiv 9 \pmod {16}$, is representable by exactly one of the quadratic forms.
\end{theorem}

\noindent Kaplansky proved his theorem using two well-known results:  {\em $2$ is a $4$th power modulo a prime $p$ if and only if $p$ is represented by $x^2+64y^2$ (Gauss \cite[p. 530]{G})} and {\em $-4$ is an $8$th power modulo a prime $p$ if and only if $p$ is represented by $x^2+32y^2$ (Barrucand and Cohn \cite{BC})}.  Using class field theory, Brink \cite{B1} was able to prove five more theorems similar to that of Kaplansky.  Three of which are

\begin{theorem}\cite[Theorem 1]{B1}\label{theorem:Brink-Thm1} A prime $p\equiv 1 \pmod {20}$ is representable by both or none of $x^2+20y^2$ and $x^2+100y^2$, whereas a prime $p\equiv 9 \pmod {20}$ is representable by exactly one of these forms.
\end{theorem}

\begin{theorem}\cite[Theorem 4]{B1}\label{theorem:Brink-Thm4} A prime $p\equiv 1,65,81 \pmod {112}$ is representable by both or none of $x^2+14y^2$ and $x^2+448y^2$, whereas a prime $p\equiv 9,25,57 \pmod {112}$ is representable by exactly one of these forms.
\end{theorem}

\begin{theorem}\cite[Theorem 5]{B1}\label{theorem:Brink-Thm5} A prime $p\equiv 1,169 \pmod {240}$ is representable by both or none of $x^2+150y^2$ and $x^2+960y^2$, whereas a prime $p\equiv 49,121 \pmod {240}$ is representable by exactly one of these forms.
\end{theorem}

\noindent In \cite{B1}, Brink claims that these are the only results of their kind and gives a heuristic argument as support.  As an example, Brink shows that there is no similar result for primes represented by $x^2+128y^2$ and $x^2+256y^2$.  In \cite{B2}, Brink gives elementary proofs of Theorems \ref{theorem:Kap-Thm} and \ref{theorem:Brink-Thm1} and also shows that Theorem \ref{theorem:Kap-Thm} is equivalent to a result of Glaisher \cite{Gl}:
{\em Let $p$ be an odd prime and let $h$ and $h^{\prime}$ be the class numbers corresponding to the discriminants $-4p$ and $-8p$ respectively.  If $p\equiv 1 \pmod {16}$, then either both or none of $h$ and $h^{\prime}$ are divisible by $8$; if $p\equiv 9 \pmod {16}$, then exactly one of these class numbers is divisible by $8$.}

It turns out that there are at least two pairs of discriminants for Kaplansky-like theorems on principal binary quadratic forms that are not on Brink's list.  Our two new results read
\begin{theorem}\label{theorem:NewKap1}  A prime $p\equiv 1 \pmod {48}$, is representable by both or none of $x^2+64y^2$ and $x^2+288y^2,$ whereas a prime $p\equiv 25 \pmod {48}$, is representable by exactly one of the quadratic forms.
\end{theorem}

\begin{theorem}\label{theorem:NewKap2} A prime $p\equiv 1,65,81 \pmod {112}$ is representable by both or none of $x^2+56y^2$ and $x^2+448y^2$, whereas a prime $p\equiv 9,25,57 \pmod {112}$ is representable by exactly one of these forms.
\end{theorem}

Prior to the work of Kaplansky and Brink, Muskat \cite{Mu} used Dirichlet's method of proof of the above mentioned theorem of Gauss to prove results on the simultaneous representations of primes by binary quadratic forms, with several of Muskat's results being related to Brink's five theorems.  Indeed, Williams \cite{W2} pointed out that Brink's results follow from Muskat's.   In particular, Williams showed that \cite[Theorem 1]{Mu} implies \cite[Theorem 1]{B1}.  Among the corollaries to Muskat's general results, we find the following:

\begin{theorem}\cite[Theorem 1]{Mu}\label{theorem:Mu-Thm1}  Let $p\equiv 1 \textup{ or } 9 \pmod {20}$ be a prime; then we may write $p=M^2+N^2$ with $N$ even and $M\equiv 1 \pmod 4$, and $p=A^2+5B^2$.  Noting that either $M$ or $N$ is divisible by $5$, we conclude that:
\begin{itemize}
\item[(a)] if $p\equiv 1 \pmod {20}$, then $A$ is even if and only if $5|M$,
\item[(b)] if $p\equiv 9 \pmod {20}$, then $A$ is even if and only if $5|N$.
\end{itemize}
\end{theorem}
\begin{theorem}\cite[Theorem 3]{Mu}\label{theorem:Mu-Thm3}  Let $p=M^2+7N^2$ be a prime $\equiv \pm 1 \pmod {8}$ with $M$ or $N\equiv 1 \pmod {4}$.  Then
\begin{itemize}
\item[(a)]  $p=A^2+14B^2$ is solvable if and only if $2p+M+N\equiv 3 \pmod 8$,
\item[(b)]  $p=7C^2+2D^2$ is solvable if and only if $2p+M+N\equiv 7 \pmod 8$.
\end{itemize}
\end{theorem}

In this paper, we will prove our results by putting them into the context of threefield identities as found in Andrews, Dyson, and Hickerson \cite{ADH} and Cohen \cite{C}.  We quickly review the two types of threefield identities found in \cite{ADH}.  Let $D$ and $E$ be distinct squarefree integers not equal to $1$, and let $F$ be the squarefree part of $DE$.  There is then an identity between representations of odd integers $n$, for which the Jacobi symbols $(D/n)$, $(E/n)$, and $(F/n)$ are equal to $1$, by quadratic forms associated with the fields $\mathbb{Q}(\sqrt{D})$, $\mathbb{Q}(\sqrt{E})$, and $\mathbb{Q}(\sqrt{F})$.  The first type of identity comes from the case in which $D$, $E$, and $F$ are all positive, and the second type of identity comes from the case in which two of the integers, say $D$ and $E$, are negative and one, say $F$, is positive.  For the second case, the generating functions turn out to be theta functions, and the identity equates two theta functions, expressed in terms of $J$'s, and a Hecke-type sum whose weight system depends on which of the two angular regions of the plane is being summed over.  Examples of the second type will be the subject of this paper.  The generating functions for the first case are not theta functions; the identity equates three Hecke-type sums, whose weight systems do not depend on the angular regions.  An example of the second type would be the three ways of writing the function $\sigma(q)$ of \cite{ADH}.  Cohen \cite{C} reinterpreted the results of \cite{ADH} in terms of a $q$-identity for a certain Maass waveform.

We cover preliminaries on theta functions, Appell-Lerch sums, and Hecke-type double sums in Section \ref{section:Preliminaries}.  In Section \ref{section:Kap-Thm}, we give a new proof of Kaplansky's Theorem \ref{theorem:Kap-Thm} using the identity
\begin{equation*}
J_{1,4}J_{2,4}=J_{1,2}\overline{J}_{1,4}=\frac{1}{2}\Big ( \sum_{\substack{r-2s\ge 0\\r+2s\ge 0}}-\sum_{\substack{r-2s< 0\\r+2s<0}}\Big )(-1)^{r+s}q^{[(2r+1)^2-2(2s)^2-1]/8}.
\end{equation*}
Although we only need the first equality, the entire threefield identity is included so that the interested reader can see how it underlies the main result of Barrucand and Cohn \cite{BC}.  For the remainder of our results, we will omit discussion of the Hecke-type sums.  In Section \ref{section:B1-Mu1} we prove Brink's Theorem \ref{theorem:Brink-Thm1} and Muskat's Theorem \ref{theorem:Mu-Thm1} using
\begin{equation*}
J_{1,5}J_{2,5}=J_1J_5.
\end{equation*}
In Section \ref{section:NewKap}, we prove the new Kaplansky-like Theorem \ref{theorem:NewKap1} using the identity
\begin{equation*}
\overline{J}_{1,3}j(q;-q^3)=J_{6,12}\overline{J}_{2,6}.
\end{equation*}
In Section \ref{section:NewKap2-Mu3}, we prove the new Kaplansky-like Theorem \ref{theorem:NewKap2} using
\begin{align*}
q\overline{J}_{2,8}j(q^7;-q^7)+J_{1,4}J_{7,14}&=\overline{J}_{1,4}J_{14,28}
\end{align*}
 and prove Muskat's Theorem \ref{theorem:Mu-Thm3} using the above identity and three identities similar to it.

\section{Preliminaries}\label{section:Preliminaries}

We will frequently use the following identities without mention.  They easily follow from the definitions.
\begin{subequations}
\begin{equation*}
\overline{J}_{0,1}=2\overline{J}_{1,4}=\frac{2J_2^2}{J_1},  \overline{J}_{1,2}=\frac{J_2^5}{J_1^2J_4^2},   J_{1,2}=\frac{J_1^2}{J_2},  \overline{J}_{1,3}=\frac{J_2J_3^2}{J_1J_6},  J_{1,4}=\frac{J_1J_4}{J_2},
\end{equation*}
\begin{equation*}
J_{1,6}=\frac{J_1J_6^2}{J_2J_3},   \overline{J}_{1,6}=\frac{J_2^2J_3J_{12}}{J_1J_4J_6}.
\end{equation*}
\end{subequations}
Also following from the definitions are the following general identities \cite{HM}:
\begin{subequations}
\begin{equation}
j(q^n x;q)=(-1)^nq^{-\binom{n}{2}}x^{-n}j(x;q), \ \ n\in\mathbb{Z},\label{equation:1.8}
\end{equation}
\begin{equation}
j(x;q)=j(q/x;q)=-xj(x^{-1};q)\label{equation:1.7},
\end{equation}
\begin{equation}
j(-x;q)={J_{1,2}j(x^2;q^2)}/{j(x;q)} \ \ {\text{if $x$ is not an integral power of $q$,}}\label{equation:1.9}
\end{equation}
\begin{equation}
j(x;q)={J_1}j(x;q^n)j(qx;q^n)\cdots j(q^{n-1}x;q^n)/{J_n^n} \ \ {\text{if $n\ge 1$,}}\label{equation:1.10}
\end{equation}
\begin{equation}
j(x;-q)={j(x;q^2)j(-qx;q^2)}/{J_{1,4}},\label{equation:1.11}
\end{equation}
\begin{equation}
j(x^n;q^n)={J_n}j(x;q)j(\zeta_nx;q)\cdots j(\zeta_n^{n-1}x;q^n)/{J_1^n} \ \ {\text{if $n\ge 1$, $\zeta_n$ $n$-th primitive root,}}\label{equation:1.12}
\end{equation}
\begin{equation}
j(z;q)=\sum_{k=0}^{m-1}(-1)^k q^{\binom{k}{2}}z^k
j\big ((-1)^{m+1}q^{\binom{m}{2}+mk}z^m;q^{m^2}\big ),\label{equation:jsplit}
\end{equation}
\end{subequations}

\noindent where $z$ is not an integral power of $q$.

More useful theta functions identities are, see for example \cite{HM}:
\begin{proposition}   For generic $x,y,z\in \mathbb{C}^*$ 
\begin{subequations}
\begin{equation}
j(qx^3;q^3)+xj(q^2x^3;q^3)=j(-x;q)j(qx^2;q^2)/J_2={J_1j(x^2;q)}/{j(x;q)},\label{equation:H1Thm1.0}
\end{equation}
\begin{equation}
j(x;q)j(y;q)=j(-xy;q^2)j(-qx^{-1}y;q^2)-xj(-qxy;q^2)j(-x^{-1}y;q^2),\label{equation:H1Thm1.1}
\end{equation}
\begin{equation}
j(-x;q)j(y;q)-j(x;q)j(-y;q)=2xj(x^{-1}y;q^2)j(qxy;q^2),\label{equation:H1Thm1.2A}
\end{equation}
\begin{equation}
j(-x;q)j(y;q)+j(x;q)j(-y;q)=2j(xy;q^2)j(qx^{-1}y;q^2),\label{equation:H1Thm1.2B}
\end{equation}
\begin{equation}
j(x;q)j(y;q^n)=\sum_{k=0}^n(-1)^kq^{\binom{k}{2}}x^kj\big ((-1)^nq^{\binom{n}{2}+kn}x^ny;q^{n(n+1)}\big )j\big (-q^{1-k}x^{-1}y;q^{n+1} \big ).\label{equation:Thm1.3AH6}
\end{equation}
\end{subequations}
\end{proposition}

\noindent Identity (\ref{equation:H1Thm1.0}) is the quintuple product identity.

 We will use the following definition of an Appell-Lerch sum.  
\begin{definition}  \label{definition:mdef} Let $x,z\in\mathbb{C}-{0}$ with neither $z$ nor $xz$ an integral power of $q$. Then
\begin{equation}
m(x,q,z):=\frac{1}{j(z;q)}\sum_r\frac{(-1)^rq^{\binom{r}{2}}z^r}{1-q^{r-1}xz}.\label{equation:H2.0}
\end{equation}
\end{definition}
\noindent These sums were first studied by Appell \cite{Ap} and then by Lerch \cite{L1}.  We will use the following definition of the building block of Hecke-type double sums and its basic properies \cite{HM}:  
\begin{definition} \label{definition:fabc-def}  Let $x,y\in\mathbb{C}-\{0\}$ and define $\sg (r):=1$ for $r\ge 0$ and $\sg(r):=-1$ for $r<0$. Then
\begin{equation*}
f_{a,b,c}(x,y,q):=\sum_{\substack{\sg (r)=\sg(s)}} \sg(r)(-1)^{r+s}x^ry^sq^{a\binom{r}{2}+brs+c\binom{s}{2}}.\\
\end{equation*}
\end{definition}

\begin{proposition} \label{proposition:fabc-prop} For $x,y\in\mathbb{C}-\{0\}$
\begin{align}
f_{a,b,c}(x,y,q)&=f_{a,b,c}(-x^2q^a,-y^2q^c,q^4)-xf_{a,b,c}(-x^2q^{3a},-y^2q^{c+2b},q^4)\label{equation:fabc-mod2}\\
&\ \ \ \ -yf_{a,b,c}(-x^2q^{a+2b},-y^2q^{3c},q^4)+xyq^bf_{a,b,c}(-x^2q^{3a+2b},-y^2q^{3c+2b},q^4).\notag\\
f_{a,b,c}(x,y,q)&=-\frac{q^{a+b+c}}{xy}f_{a,b,c}(q^{2a+b}/x,q^{2c+b}/y,q).\label{equation:H7eq1.14}
\end{align}
\end{proposition}

To relate Hecke-type double sums to Appell-Lerch sums and theta functions, we use the $n=1$ specialization of Theorem $0.9$ of [HM]:
\begin{proposition} \label{proposition:f131}For generic $x,y\in \mathbb{C} - \{0\}$
\begin{align}
f_{1,3,1}(x,y,q)=j(y;q)&m(-q^5x/y^3,q^8,y/x)+j(x;q)m(-q^5y/x^3,q^8,x/y)\label{equation:f131}\\
&\ \ \ \ -\frac{qxyJ_{2,4}J_{8,16}j(q^{3}xy;q^8)j(q^{14}x^{2}y^{2};q^{16})}
{j(-q^{3}x^2;q^8)j(-q^{3}y^2;q^8)}.\notag
\end{align}
\end{proposition}

\begin{theorem}\label{theorem:master-id-theorem} Suppose that $x$, $q$, $A,$ and $B$ are nonzero complex numbers with $|q|<1$, that $r$, $t$, $m$, $n$, $a$, and $b$ are integers with $m$ and $n$ positive and $ra+tb=1,$ and that $M$ is a positive integer divisible by $(mt^2+nr^2)/\textup{gcd}(mt,nr)$.  Then
\begin{align}
j(Ax^r,q^m)j(Bx^t,q^n)&=\sum_{i=0}^{M-1}x^i(-A)^{ai}(-B)^{bi}q^{m\binom{ai}{2}+n\binom{bi}{2}}\label{equation:master-id-theorem}\\
&\ \ \ \ \ \cdot j(-(-A)^t(-B)^{-r}q^{m\binom{t}{2}+n\binom{r+1}{2}+(mat-nbr)i};q^{mt^2+nr^2})\notag\\
&\ \ \ \ \ \cdot j(-(-A)^{\frac{nrM}{mt^2+nr^2}}(-B)^{\frac{mtM}{mt^2+nr^2}}x^Mq^{\frac{mn(2i+M-r-t)M}{2(mt^2+nr^2)}};q^{\frac{mnM^2}{mt^2+nr^2}}).\notag
\end{align}
\end{theorem}
\begin{proof}[Proof of Theorem \ref{theorem:master-id-theorem}] Denote by $f(x)$ the left-hand side of (\ref{equation:master-id-theorem}).  We assume that gcd$(r,t)=1$, if not, we can substitute $x\rightarrow x^{\gcd(r,t)}$ into the left-hand side of (\ref{equation:master-id-theorem}). Let $a$ and $b$ be integers such that
\begin{equation}
ra+tb=1.\label{equation:proof-step1}
\end{equation}
First we determine the coefficient of $x^i$ in $f(x)$ for all $i$.  We have 
\begin{align}
f(x)&=\sum_{k}(-Ax^r)^kq^{m\binom{k}{2}}\sum_{\ell}(-Bx^t)^{\ell}q^{n\binom{\ell}{2}}\label{equation:proof-step2}\\
&=\sum_{k,\ell}(-A)^k(-B)^{\ell}x^{rk+t\ell}q^{m\binom{k}{2}+n\binom{\ell}{2}}.\notag
\end{align}
To find the coefficient of $x^i$, we need to consider representations of $i$ in the form $rk+t\ell$.  One such representation is $i=r(ai)+t(bi)$; all others are obtained by adding a multiple of $t$ to $ai$ and subtracting the corresponding multiple of $r$ from $bi$.  I.e. we must have $k=ai+pt$ and $\ell=bi-pr$ for some integer $p.$  So 
\begin{align}
\textup{coefficient of $x^i$ in $f(x)$}&=\sum_p(-A)^{ai+pt}(-B)^{bi-pr}q^{m\binom{ai+pt}{2}+n\binom{bi-pr}{2}}\label{equation:proof-step3}\\
=(-A)^{ai}(-B)^{bi}&q^{m\binom{ai}{2}+n\binom{bi}{2}}j(-(-A)^t(-B)^{-r}q^{m\binom{t}{2}+n\binom{r+1}{2}+(mat-nbr)i};q^{mt^2+nr^2})\notag
\end{align}
Summing over $i$, we have
\begin{align}
f(x)&=\sum_ix^i(-A)^{ai}(-B)^{bi}q^{m\binom{ai}{2}+n\binom{bi}{2}}j(-(-A)^t(-B)^{-r}q^{m\binom{t}{2}+n\binom{r+1}{2}+(mat-nbr)i};q^{mt^2+nr^2})\notag\\
&=\sum_{i}T(i),\label{equation:proof-step4}
\end{align}
say.  Next we combine terms in this sum for which the $j$'s are related in the way that $j(x,q)$ and $j(q^kx,q)$ are related.  Changing $i$ by $M$ changes the exponent of $q$ in the first parameter of $j$ in (\ref{equation:proof-step4}) by $(mat-nbr)M$.  We need this change to be a multiple of $mt^2+nr^2$.  In other words, $M$ should be divisible by
\begin{align}
(mt^2+nr^2)/\text{gcd}(mt^2+nr^2,mat-nbr).\label{equation:proof-step5}
\end{align}
We can rewrite this without the $a$ and $b:$  since $\textup{gcd}(rmt)=1$,
\begin{align}
\textup{gcd}(m&t^2+nr^2,mat-nbr)\label{equation:proof-step6}\\
&=\text{gcd}(mt^2+nr^2,r(mat-nbr),t(mat-nbr))\notag\\
&=\text{gcd}(mt^2+nr^2,r(mat-nbr)+b(mt^2+nr^2),t(mat-nbr)-a(mt^2+nr^2))\notag\\
&=\text{gcd}(mt^2+nr^2,mart+mbt^2,-(nbrt+nar^2))\notag\\
&=\text{gcd}(mt^2+nr^2,mt(ar+bt),-nr(bt+ar))\notag\\
&=\text{gcd}(mt^2+nr^2,mt,nr)\notag\\
&=\text{gcd}(mt,nr).\notag
\end{align}
So we need:
\begin{equation}
M \textup{ is divisible by } (mt^2+nr^2)/\textup{gcd}(mt,nr).\label{equation:proof-step7}
\end{equation}
Then, for any integer $k$,
\begin{align}
T(i+kM)/T(i)=(-A)^{\frac{knrM}{mt^2+nr^2}}(-B)^{\frac{kmtM}{mt^2+nr^2}}x^{kM}q^{\frac{kmn(2i+kM-r-t)M}{2(mt^2+nr^2)}}.\label{equation:proof-step8}
\end{align}
Now we can rewrite the index $i$ in (\ref{equation:proof-step4}) as $i+kM$, where $i$ ranges through a complete residue system mod $M$ and $k$ ranges over all integers:
\begin{align}
f(x)&=\sum_{i=0}^{M-1}\sum_kT(i+kM)\label{equation:proof-step9}\\
&=\sum_{i=0}^{M-1}T(i)\sum_k (-A)^{\frac{knrM}{mt^2+nr^2}}(-B)^{\frac{kmtM}{mt^2+nr^2}}x^{kM}q^{\frac{kmn(2i+kM-r-t)M}{2(mt^2+nr^2)}}\notag\\
&=\sum_{i=0}^{M-1}T(i)j(-(-A)^{\frac{nrM}{mt^2+nr^2}}(-B)^{\frac{mtM}{mt^2+nr^2}}x^Mq^{\frac{mn(2i+M-r-t)M}{2(mt^2+nr^2)}};q^{\frac{mnM^2}{mt^2+nr^2}}).\notag
\end{align}
Finally, we substitute $T(i)$ from $(\ref{equation:proof-step4})$ into (\ref{equation:proof-step9}).
\end{proof}
\begin{proposition}\label{proposition:master-prop} We have
\begin{align}
j(x;q)j(x;q^7)=-\frac{q}{x}j(x;q^2) j(x^3;q^{14})-\frac{x}{q}j(qx;q^2) j(q^7x^3;q^{14})+\frac{x}{q}J_{1,2}j(q^7x^4;q^{14}).\label{equation:master-id}
\end{align}
\end{proposition}
\begin{proof}[Proof of Proposition \ref{proposition:master-prop}] 
Four applications of Theorem \ref{theorem:master-id-theorem} change (\ref{equation:master-id}) into a sum of $52$ terms of the form
\begin{equation*}
q^ax^bj(\pm q^c;q^d)j(-q^{14e}x^{16};q^{224}),
\end{equation*}
for various integers $a$, $b$, $c$, $d$, and $e$.  We can apply (\ref{equation:1.8}) to reduce all of the $e$'s to the range $[0,15]$.  Then, for each $e$, the coefficient of
\begin{equation*}
j(-q^{14e}x^{16};q^{224})
\end{equation*}
is a sum of terms of the form
\begin{equation*}
q^ax^bj(\pm q^c;q^d),
\end{equation*}
and we can prove that it equals zero by some application of (\ref{equation:jsplit}).  For example, the coefficient of 
\begin{equation*}
j(-q^{14}x^{16};q^{224})
\end{equation*}
is
\begin{equation*}
q^{21}x^{-6}\Big [ j(-q;q^8)-qj(-q^6;q^{32})-j(-q^{22};q^{32})\Big ],
\end{equation*}
which by (\ref{equation:jsplit}) with $q\rightarrow q^8$,$x= -q$, and $n=2$ equals zero.
\end{proof}

\begin{corollary}\label{corollary:NewKapMu}  The following identities are true:
\begin{align}
q\overline{J}_{2,8}j(q^7;-q^7)+J_{1,4}J_{7,14}&=\overline{J}_{1,4}J_{14,28},\label{equation:NewKap2-A}\\
\overline{J}_{2,8}j(q^5;-q^7)+qJ_{1,4}J_{1,14}&=\overline{J}_{1,4}J_{10,28},\label{equation:Mu-Thm3B1-A}\\
\overline{J}_{2,8}j(q^3;-q^7)-J_{1,4}J_{5,14}&=q\overline{J}_{1,4}J_{6,28},\label{equation:Mu-Thm3B2-A}\\
\overline{J}_{2,8}j(q;-q^7)-J_{1,4}J_{3,14}&=q^2\overline{J}_{1,4}J_{2,28}.\label{equation:Mu-Thm3B3-A}
\end{align}
\end{corollary}
\begin{proof}[Proof of Corollary \ref{corollary:NewKapMu}]  We first replace $q$ by $q^2$ and $x$ by $-q^k$ in (\ref{equation:master-id}):
\begin{equation}
\overline{J}_{k,2}\overline{J}_{k,14}=q^{2-k}\overline{J}_{k,4}\overline{J}_{3k,28}+q^{k-2}\overline{J}_{k+2,4}\overline{J}_{3k+14,28}-q^{k-2}J_{2,4}J_{4k+14,28}.\label{equation:step1}
\end{equation}
When $k$ is odd,
\begin{equation}
\overline{J}_{k,4}=q^{-(k-1)(k-3)/8}\overline{J}_{1,4}.\label{equation:step2}
\end{equation}
We can apply this to the first two terms of the right side of (\ref{equation:step1}) and then combine them using (\ref{equation:jsplit}) with $m=2$ to obtain

\begin{equation}
\overline{J}_{k,2}\overline{J}_{k,14}=q^{-1}\overline{J}_{1,4}\overline{J}_{2k,7}-q^{k-2}J_{2,4}J_{4k+14,28}.\label{equation:step4}
\end{equation}
Now replace $q$ by $-q$:
\begin{equation}
J_{k,2}J_{k,14}=-q^{-1}J_{1,4}j(-q^{2k};-q^7)+q^{k-2}J_{2,4}J_{4k+14,28}.\label{equation:step5}
\end{equation}
When $k=1$, this becomes
\begin{equation}
J_{1,2}J_{1,14}=-q^{-1}J_{1,4}j(-q^{2};-q^7)+q^{-1}J_{2,4}J_{18,28}.\label{equation:step6}
\end{equation}
Multiply by $J_{1,4}/J_{1,2}=\overline{J}_{2,8}/J_{1,4}=\overline{J}_{1,4}/J_{2,4}$ and rearrange terms to get (\ref{equation:Mu-Thm3B1-A}).  Similarly, letting $k=3$ in (\ref{equation:step5}) gives (\ref{equation:Mu-Thm3B3-A}),  $k=5$ gives (\ref{equation:Mu-Thm3B2-A}),  and $k=7$ gives (\ref{equation:NewKap2-A}).
\end{proof}

\section{Proof of Theorem \ref{theorem:Kap-Thm}.}\label{section:Kap-Thm}
For this threefield identity, we have D=-1, E=-2, F=2.  We first claim the following for any given $k\ge 0$.  The excess of the number of inequivalent solutions of $8k+1=x^2+y^2$ ($x>0$ odd) in which
\begin{align*}
x\equiv \pm 1 \pmod{8}, \ y\equiv 0 \pmod{8}, {\text { or }} x\equiv \pm 3 \pmod{8}, \ y\equiv 4 \pmod{8}, 
\end{align*}
over those in which 
\begin{align*}
x\equiv \pm 3 \pmod{8}, \ y\equiv 0 \pmod{8}, {\text { or }} x\equiv \pm 1 \pmod{8}, \ y\equiv 4 \pmod{8},
\end{align*}
equals the excess of the number of inequivalent solutions of $8k+1=x^2+2y^2$ ($x>0$ odd) in which $y\equiv 0 \pmod{4}$ over those in which $y\equiv 2 \pmod{4}$ equals the excess of the number of inequivalent solutions of $8k+1=x^2-2y^2$  ($x>0$ odd, $-2x/4\le y < 2x/4$, see \cite[Lemma $3$]{ADH} with fundamental unit $3+2\sqrt{2}$) in which
\begin{align*}
&x\equiv 1 \pmod{4}, \ y\equiv 0 \pmod{4}, {\text { or }} x\equiv 3 \pmod{4}, \ y\equiv  2 \pmod{4}, 
\end{align*}
over those in which
\begin{align*}
&x\equiv 1 \pmod{4}, \ y\equiv  2 \pmod{4}, {\text { or }} x\equiv 3 \pmod{4}, \ y\equiv 0 \pmod{4}.
\end{align*}

We show that in terms of generating functions, this is equivalent to
\begin{equation}
J_{1,4}J_{2,4}=J_{1,2}\overline{J}_{1,4}=f_{1,3,1}(q^{3/4},-q^{3/4},-q^{1/2}).\label{equation:3field-1}
\end{equation}

\noindent The weighted set of solutions for $8k+1=x^2+y^2$ yields $J_{1,4}J_{2,4}$.  We immediately have
\begin{align*}
\sum_{r,s}&q^{\big ( (8r+1)^2+(8s)^2-1 \big )/8}+\sum_{r,s}q^{\big ( (8r+5)^2+(8s+4)^2-1 \big )/8} -\sum_{r,s}q^{\big ( (8r+5)^2+(8s)^2-1 \big )/8}\\
&\ \ \ \ \ \ \ \ \ \ -\sum_{r,s}q^{\big ( (8r+1)^2+(8s+4)^2-1 \big )/8}\\
&=\sum_{r,s}(-1)^sq^{\big ( (8r+1)^2+(4s)^2-1 \big )/8}-\sum_{r,s}(-1)^sq^{\big ( (8r+5)^2+(4s)^2-1 \big )/8}\\
&=\sum_{r,s}(-1)^{r+s}q^{\big ( (4r+1)^2+(4s)^2-1 \big )/8}\\
&=\sum_{r}(-1)^{r}q^{4\binom{r}{2}+3r}\sum_{s}(-1)^{s}q^{4\binom{s}{2}+2s}
={J}_{1,4}J_{2,4}.&(\text{by }(\ref{equation:JTPid}))
\end{align*}
\noindent The weighted set of solutions for $8k+1=x^2+2y^2$ yields $J_{1,2}\overline{J}_{1,4}$.  We have
\begin{align*}
\frac{1}{2}\sum_{r,s}&q^{\big ( (2r+1)^2+2(4s)^2-1 \big )/8}-\frac{1}{2}\sum_{r,s}q^{\big ( (2r+1)^2+2(4s+2)^2-1 \big )/8}\\
&=\frac{1}{2}\sum_{r,s}(-1)^sq^{\big ( (2r+1)^2+2(2s)^2-1 \big )/8}\\
&=\frac{1}{2}\sum_{r}q^{\binom{r}{2}+r}\sum_{s}(-1)^sq^{2\binom{s}{2}+s}=\frac{1}{2}\overline{J}_{0,1}{J}_{1,2}=\overline{J}_{1,4}{J}_{1,2}.&(\text{by }(\ref{equation:JTPid}))
\end{align*}
The weighted set of solutions for $8k+1=x^2-2y^2$ yields

\begin{align}
&\sum_{\substack{r\ge 0 \\ -(4r+1)\le 2\cdot 4s<4r+1}}q^{[(4r+1)^2-2(4s)^2-1]/8}+\sum_{\substack{r\ge 0 \\ -(4r+3)\le 2\cdot (4s+2)<4r+3}}q^{[(4r+3)^2-2(4s+2)^2-1]/8}\notag\\
& \ \ \ \ \ -\sum_{\substack{r\ge 0 \\ -(4r+1)\le 2\cdot (4s+2)<4r+1}}q^{[(4r+1)^2-2(4s+2)^2-1]/8}-\sum_{\substack{r\ge 0 \\ -(4r+3)\le 2\cdot 4s<4r+3}}q^{[(4r+3)^2-2(4s)^2-1]/8}\notag\\
=&\sum_{\substack{r\ge 0 \\ -(2r+1)\le 2\cdot 4s<2r+1}}(-1)^rq^{[(2r+1)^2-2(4s)^2-1]/8}
-\sum_{\substack{r\ge 0 \\ -(2r+1)\le 2\cdot (4s+2)<2r+1}}(-1)^rq^{[(2r+1)^2-2(4s+2)^2-1]/8}\notag\\
=&\sum_{\substack{r\ge 0 \\ -(2r+1)\le 2\cdot 2s<2r+1}}(-1)^{r+s}q^{[(2r+1)^2-2(2s)^2-1]/8}
.\label{equation:Id1-A}
\end{align}
Using the substitution $r=-1-r$, we rewrite (\ref{equation:Id1-A}):
\begin{align}
&\frac{1}{2}\sum_{\substack{r\ge 0 \\ -(2r+1)\le 2\cdot 2s<2r+1}}(-1)^{r+s}q^{[(2r+1)^2-2(2s)^2-1]/8}\notag-\frac{1}{2}\sum_{\substack{r< 0 \\ (2r+1)\le 2\cdot 2s<-2r-1}}(-1)^{r+s}q^{[(2r+1)^2-2(2s)^2-1]/8}\notag\\
&=\frac{1}{2}\Big ( \sum_{\substack{2r-4s+1>0\\2r+4s+1\ge 0}}-\sum_{\substack{2r-4s+1\le 0\\2r+4s+1<0}}\Big )(-1)^{r+s}q^{[(2r+1)^2-2(2s)^2-1]/8}\notag\\
&=\frac{1}{2}\Big ( \sum_{\substack{r-2s\ge 0\\r+2s\ge 0}}-\sum_{\substack{r-2s< 0\\r+2s<0}}\Big )(-1)^{r+s}q^{[(2r+1)^2-2(2s)^2-1]/8}.\label{equation:Id1-B}
\end{align}
If we let $u=r+2s$ and $v=r-2s$ and sum over $(u,v)$, we must have that $r=(u+v)/2$ and $s=(u-v)/4$ where $u\equiv v \pmod 4$.  So we can write (\ref{equation:Id1-B}) as
\begin{align}
\frac{1}{2}\Big ( f_{1,3,1}(q^2,q^2,q^2)&-qf_{1,3,1}(q^4,q^4,q^2)+q^3f_{1,3,1}(q^6,q^6,q^2)-q^6f_{1,3,1}(q^8,q^8,q^2)\Big )\label{equation:Id1-C}\\
&=f_{1,3,1}(q^2,q^2,q^2)+q^3f_{1,3,1}(q^6,q^6,q^2)&\text{(by (\ref{equation:H7eq1.14}))}\notag\\
&=f_{1,3,1}(q^{3/4},-q^{3/4},-q^{1/2}).&\text{(by (\ref{equation:fabc-mod2}))}\notag
\end{align}

The first equality of (\ref{equation:3field-1}) follows from a simple product rearrangement.  For the second equality  of (\ref{equation:3field-1}), use Proposition \ref{proposition:f131} to see that (\ref{equation:Id1-C}) can be evaluated as 
\begin{align*}
f_{1,3,1}(q^{3/4},-q^{3/4},-q^{1/2})&=\Big [j(-q^{3/4};-q^{1/2})+j(q^{3/4};-q^{1/2}) \Big ]m(-q,q^4,-1)\\
&\ \ \ \ \ -\frac{q^2J_{1,2}J_{4,8}J_{3,4}J_{10,8}}{J_{3,4}^2}.
\end{align*}
Using (\ref{equation:jsplit}), we note that $j(x;q)=j(-qx^2;q^4)-xj(-q^3x^2;q^4)$.  It follows that the bracketed expression vanishes yielding
\begin{align}
f_{1,3,1}(q^{3/4},-q^{3/4},-q^{1/2})=-\frac{q^2J_{1,2}J_{4,8}J_{3,4}J_{10,8}}{J_{3,4}^2}
=J_{1,2}\overline{J}_{1,4}.
\end{align}

\begin{proof}[Proof of Theorem  \ref{theorem:Kap-Thm}]  With a simple change of variables, we can rewrite the weights of solutions of $8k+1=x^2+y^2$ and $8k+1=x^2+2y^2$ in terms of the weights of solutions of $8k+1=x^2+16y^2$ and $8k+1=x^2+8y^2$.  We then have that the excess of the number of solutions of $8k+1=x^2+16y^2$ $(x>0)$ with 
\begin{equation*}
 x\equiv \pm 1 \pmod{8}, \ y\textup{ even}, \text{ or }x\equiv \pm 3 \pmod{8}, \ y\textup{ odd } 
\end{equation*}
over the number with
\begin{equation*}
x\equiv \pm 3 \pmod{8},\  y\textup{ even}, \text{ or } x\equiv \pm 1 \pmod{8}, \ y\textup{ odd } 
\end{equation*}
equals the number of excess of solutions of $8k+1=x^2+8y^2$ $(x>0)$ with $y$ even over the number with $y$ odd.

If $8k+1$ is prime, then there are exactly two representations by each of these quadratic forms, with one obtained from the other by negating $y$.  So if $p\equiv 1 \pmod{8}$ is prime, the $p$'s unique representation of the form $x^2+16y^2$ $(x>0, \ y>0)$ has
\begin{equation*}
 x\equiv \pm 1 \pmod{8}, \ y\textup{ even}, \text{ or }  x\equiv \pm 3 \pmod{8}, \ y\textup{ odd }
\end{equation*}
if and only if $p$'s unique representation of the form  $x^2+8y^2$ $(x>0, \ y>0)$ has $y$ even, i.e., iff $p$ has a representation of the form $x^2+32y^2.$

We now consider the two possibilities for $p$ mod $16$:

If $p\equiv 1 \pmod{16}$, then, in the representation $p=x^2+16y^2$, we must have that $x\equiv \pm 1 \pmod{8}$.  Thus, $p$'s representation in this form has $y$ even iff $p$ has a representation of the form $x^2+32y^2.$  In other words, $p$ has a representation of the form $x^2+64y^2$ iff $p$ has a representation of the form $x^2+32y^2.$

If $p\equiv 9 \pmod{16}$, then, in the representation $p=x^2+16y^2$, we must have that $x\equiv \pm 3 \pmod{8}$.  Thus, $p$'s representation in this form has $y$ odd iff $p$ has a representation of the form $x^2+32y^2.$  In other words, $p$ does not have a representation of the form $x^2+64y^2$ iff $p$ has a representation of the form $x^2+32y^2.$
\end{proof}


\section{Proofs of Theorems \ref{theorem:Brink-Thm1} and \ref{theorem:Mu-Thm1}}\label{section:B1-Mu1}

For this threefield identity, we have $D=-1$, $E=-5$, $F=5$; however, we will omit the Hecke-type sum.  We claim the following for any given $k\ge 0$.  The excess of the number of inequivalent solutions of $4k+1=x^2+y^2$ ($x>0$ odd) in which
\begin{align*}
&x\equiv \pm 1 \pmod {10}, \ y\equiv 0 \pmod {10}, \ \text{or } x\equiv 5 \pmod {10}, \ y\equiv \pm 2 \pmod {10},\\
\text{or }&x\equiv 3\pmod {10}, \ y\equiv 2,6 \pmod {10}, \ \text{or }x\equiv 7\pmod {10}, \ y\equiv 4,8 \pmod {10},\\
\text{or }&x\equiv 1\pmod {10}, \ y\equiv 6 \pmod {10}, \ \text{or }x\equiv 9\pmod {10},\  y\equiv 4 \pmod {10},
\end{align*}
over those in which 
\begin{align*}
&x\equiv \pm 3 \pmod {10}, \ y\equiv 0 \pmod {10}, \ \text{or } x\equiv 5 \pmod {10}, \ y\equiv \pm 4 \pmod {10},\\
\text{or }&x\equiv 1\pmod {10}, \ y\equiv 2,4 \pmod {10}, \ \text{or }x\equiv 9\pmod {10}, \ y\equiv 6,8 \pmod {10},\\
\text{or }&x\equiv 3\pmod {10}, \ y\equiv 8 \pmod {10}, \ \text{or }x\equiv 7\pmod {10},\  y\equiv 2 \pmod {10},
\end{align*}
equals the excess of the number of inequivalent solutions of $4k+1=x^2+5y^2$ ($x\ge 0$) in which $x$ is odd and $y$ is even over those in which $x$ is even and  $y$ is odd.  For $4k+1=x^2+y^2$ ($x>0$ odd), we ignore solutions with
\begin{align*}
&x\equiv 1\pmod {10}, \ y\equiv 8\pmod {10}, \ \text{or }x\equiv 9\pmod {10}, \ y\equiv 2 \pmod {10},\\
\text{or }&x\equiv 3\pmod {10}, \ y\equiv 4 \pmod {10}, \ \text{or }x\equiv 7\pmod {10},\  y\equiv 6 \pmod {10},\\
\text{or }&x\equiv 5\pmod {10}, \ y\equiv 0 \pmod {10}.
\end{align*}

We show that in terms of generating functions, that this is equivalent to
\begin{equation}
J_{1,5}J_{2,5}=J_{1}J_{5}.\label{equation:3field-2}
\end{equation}
We rewrite the system of weights for $4k+1=x^2+y^2$ as the excess of the number inequivalent solutions in which $x+3y\equiv \pm 1 \pmod 5$ over those in which $x+3y\equiv \pm 2 \pmod 5.$  We ignore solutions with $x+3y\equiv 0 \pmod 5$.  It is then straightforward to show that the generating function is
\begin{align*}
&\overline{J}_{3,10}\overline{J}_{6,10}-q\overline{J}_{2,10}\overline{J}_{1,10}
=J_{1,5}J_{2,5}. &(\text{by }(\ref{equation:H1Thm1.1}))
\end{align*}
For the weighted set of solutions to $4k+1=x^2+5y^2$, it is straightforward to show that the generating function is
\begin{align*}
\frac{1}{2}\Big ( \overline{J}_{0,2}\overline{J}_{5,10}-q\overline{J}_{1,2}\overline{J}_{0,10}\Big )
&=\frac{1}{2}\Big [2\overline{J}_{25,60}\overline{J}_{7,12}+2q^2\overline{J}_{15,60}\overline{J}_{3,12}+2q^6\overline{J}_{5,60}\overline{J}_{1,12}\Big ]\\
&\ \ \ \ -\frac{q}{2}\Big [ 2\overline{J}_{25,60}\overline{J}_{1,12}+2q\overline{J}_{15,60}\overline{J}_{3,12}+2q^4\overline{J}_{5,60}\overline{J}_{5,12} \Big ]\\
&=\overline{J}_{25,60}J_{1,3}-q^5\overline{J}_{5,60}J_{1,3}\\
&=J_{5}J_{1},
\end{align*}
where the first equality follows from (\ref{equation:Thm1.3AH6}) with $q\rightarrow q^2$, $n= 5$, $x= -1$, $y= -q^5$  for the first bracketed expression and $q\rightarrow q^2$, $n= 5$, $x= -q$, $y= -1$ for the second bracketed expression.  The last two equalities follow from (\ref{equation:jsplit}).  Equality in (\ref{equation:3field-2}) follows from a simple product rearrangement.

\begin{proof}[Proof of Theorem \ref{theorem:Brink-Thm1}] If $4k+1$ is prime, then there are exactly two representations by each of the quadratic forms $x^2+y^2$ ($x>0$ odd) and $x^2+5y^2$ ($x\ge 0$), with one obtained from the other by negating $y$.  We now consider the two possibilities for $p \pmod{20}$:

If $p\equiv 1 \pmod {20}$, then in the representation $p=x^2+y^2$ we must have
\begin{align*}
&x\equiv \pm 1 \pmod {10}, \ y\equiv 0 \pmod {10}, \ \text{or } x\equiv 5 \pmod {10}, \ y\equiv \pm 4 \pmod {10}.
\end{align*}
The $p$'s representation in this form has $y\equiv 0 \pmod {10}$ iff $p$ has a representation of the form $x^2+20y^2.$

If $p\equiv 9 \pmod {20}$, then in the representation $p=x^2+y^2$ we must have
\begin{align*}
&x\equiv  5 \pmod {10}, \ y\equiv \pm 2 \pmod {10}, \ \text{or } x\equiv \pm 3 \pmod {10}, \ y\equiv 0 \pmod {10}.
\end{align*}
The $p$'s representation in this form has $y\equiv 0 \pmod {10}$ iff $p$ does not have a representation of the form $x^2+20y^2.$
\end{proof}
\begin{proof}[Proof of Theorem \ref{theorem:Mu-Thm1}] If $p\equiv 1 \pmod {20}$, then in the representation $p=x^2+y^2$ we must have
\begin{align*}
&x\equiv \pm 1 \pmod {10}, \ y\equiv 0 \pmod {10}, \ \text{or } x\equiv 5 \pmod {10}, \ y\equiv \pm 4 \pmod {10}.
\end{align*}
Thus $A$ is even iff $x\equiv 5 \pmod {10}, \ y\equiv \pm 4 \pmod {10}$, i.e. $5|M$. If $p\equiv 9 \pmod {20}$, then in the representation $p=x^2+y^2$ we must have
\begin{align*}
&x\equiv 5 \pmod {10}, \ y\equiv \pm 2 \pmod {10}, \ \text{or } x\equiv \pm 3 \pmod {10}, \ y\equiv 0 \pmod {10}.
\end{align*}
Thus $A$ is even iff $x\equiv \pm 3 \pmod {10}, \ y\equiv 0 \pmod {10}$, i.e. $5|N$.
\end{proof}

\section{Proof of Theorem \ref{theorem:NewKap1}}\label{section:NewKap}

For this threefield identity we use $D=-1$, $E=-2$, $F=2$ but again omit the Hecke-type sum.  We claim the following.  The excess of the number of inequivalent solutions of $12k+1=x^2+y^2$ ($x>0$ odd)  in which
\begin{align*}
&x\equiv \pm 1,\pm 7 \pmod{24}, \ y\equiv 0 \pmod{24}; x\equiv \pm 5,\pm 11 \pmod{24}, \ y\equiv 12 \pmod{24}; \\
&x\equiv \pm 3 \pmod{24}, \ y\equiv \pm 4 \pmod{24}; {\text { or }} x\equiv \pm 9 \pmod{24}, \ y\equiv \pm 8 \pmod{24},
\end{align*}
over those in which
\begin{align*}
&x\equiv \pm 5,\pm 11 \pmod{24}, \ y\equiv 0 \pmod{24}; x\equiv \pm 1,\pm 7 \pmod{24}, \ y\equiv 12 \pmod{24}; \\
&x\equiv \pm 9 \pmod{24}, \ y\equiv \pm 4 \pmod{24}; {\text { or }} x\equiv \pm 3 \pmod{24}, \ y\equiv \pm 8 \pmod{24},
\end{align*}
equals the excess of the number of inequivalent solutions of $12k+1=x^2+2y^2$ ($x>0$) in which $x\equiv \pm 1 \pmod{6}$  and $y\equiv 0 \pmod{12}$ over those in which  $x\equiv \pm 1 \pmod{6}$ and $y\equiv 6 \pmod{12}$.  We ignore solutions in the first weight system in which $x\equiv 2 \pmod 4,$ for these give the coefficents of $q^n,$ where $n$ is odd.

In terms of generating functions, this is equivalent to
\begin{equation}
\overline{J}_{1,3}j(q;-q^3)=J_{6,12}\overline{J}_{2,6}, \label{equation:3field-3}
\end{equation}
which itself holds by a simple product rearrangement.  A straight forward argument shows that the first weight system yields
\begin{align*}
J_{12,24}j(q^2;-q^6)&+q^2J_{20,24}j(-1;-q^6)\\
&=\frac{J_{6,12}\overline{J}_{6,12}}{J_{12,24}}\cdot j(q^2;-q^6)
+q^2\cdot\frac{J_{10,12}\overline{J}_{10,12}}{J_{12,24}}\cdot j(-1;-q^6)&(\text{by (\ref{equation:1.12}}))\\
&=\frac{J_{6,12}\overline{J}_{6,12}}{J_{12,24}}\cdot \frac{J_{2,12}\overline{J}_{8,12}}{J_{6,24}}
+q^2\cdot\frac{J_{10,12}\overline{J}_{10,12}}{J_{12,24}}\cdot \frac{\overline{J}_{0,12}J_{6,12}}{J_{6,24}}&(\text{by (\ref{equation:1.11}}))\\
&=\frac{J_{6,12}J_{2,12}}{J_{12,24}J_{6,24}}\cdot \Big ( \overline{J}_{6,12}\overline{J}_{8,12}+q^2\overline{J}_{10,12}\overline{J}_{0,12}\Big )\\
&=\frac{J_{6,12}J_{2,12}}{J_{12,24}J_{6,24}}\cdot \Big ( \overline{J}_{2,6}^2\Big )&(\text{by (\ref{equation:H1Thm1.1}}))\\
&=\overline{J}_{1,3}j(q;-q^3),
\end{align*}
and that the second weight system yields $J_{6,12}\overline{J}_{2,6}.$

\begin{proof}[Proof of Theorem \ref{theorem:NewKap1}]  If $p\equiv 1 \pmod {48}$ then in the representation $p=x^2+y^2$ we have
\begin{align*}
&x\equiv \pm 1,\pm 7 \pmod{24}, \ y\equiv 0 \pmod{24}, {\text { or }} x\equiv \pm 9 \pmod{24}, \ y\equiv \pm 8 \pmod{24}, \\
&x\equiv \pm 1, \pm 7 \pmod{24}, \ y\equiv 12 \pmod{24}, {\text { or }} x\equiv \pm 9 \pmod{24}, \ y\equiv \pm 4 \pmod{24}.
\end{align*}
Then $p's$ representation in this form has $y\equiv 0 \pmod 8$ iff $p=x^2+288y^2.$   

If $p\equiv 25 \pmod {48}$ then in the representation $p=x^2+y^2$ we have
\begin{align*}
&x\equiv \pm 5,\pm 11 \pmod{24}, \ y\equiv 12 \pmod{24}, {\text { or }} x\equiv \pm 3 \pmod{24}, \ y\equiv \pm 4 \pmod{24}, \\
&x\equiv \pm 5, \pm 11 \pmod{24}, \ y\equiv 0 \pmod{24}, {\text { or }} x\equiv \pm 3 \pmod{24}, \ y\equiv \pm 8 \pmod{24}.
\end{align*}
Then $p's$ representation in this form has $y\equiv 0 \pmod 8$ iff p is not of the form $x^2+288y^2.$
\end{proof}
\begin{remark}  We could also prove Theorem \ref{theorem:NewKap1} using Kaplansky's theorem on quadratic forms; however, the focus of this paper is to use threefield identities.

\end{remark}

\section{Proofs of Theorems \ref{theorem:NewKap2} and \ref{theorem:Mu-Thm3}}\label{section:NewKap2-Mu3}

Theorems \ref{theorem:NewKap2} and \ref{theorem:Mu-Thm3} follow from identities (\ref{equation:NewKap2-A})-(\ref{equation:Mu-Thm3B3-A}), which here are (\ref{equation:NewKap2})-(\ref{equation:Mu-Thm3B3}).  We state the quadratic forms and their weighted solution sets which have identities (\ref{equation:NewKap2-A})-(\ref{equation:Mu-Thm3B3-A}) as their generating functions.  Obtaining the generating functions is easy, so it will be omitted.  We then prove the two theorems.

Theorem \ref{theorem:NewKap2} and the case $p\equiv 1 \pmod 8$ of Theorem \ref{theorem:Mu-Thm3} follow from the identity:
\begin{equation}
J_{1,4}J_{7,14}+qJ_{7,28}\overline{J}_{0,2}=\overline{J}_{1,4}J_{14,28}\label{equation:NewKap2-0}
\end{equation}
or equivalently
\begin{equation}
q\overline{J}_{2,8}j(q^7;-q^7)+J_{1,4}J_{7,14}=\overline{J}_{1,4}J_{14,28}\label{equation:NewKap2}
\end{equation}

For the left-hand side of (\ref{equation:NewKap2-0}) we have two parts:
\noindent $J_{1,4}J_{7,14}$ is the generating function for the following weighted set of solutions to $8k+1=x^2+14y^2$, $k\ge 0$.  Here $x>0$. This is the excess of the number of inequivalent solutions with
\begin{align*}
x\equiv \pm 1 \pmod{8}, \ y\equiv 0 \pmod{4}, {\text { or }} x\equiv \pm 3 \pmod{8}, \ y\equiv  2 \pmod{4},
\end{align*}
over the number with
\begin{align*}
x\equiv \pm 3 \pmod{8}, \ y\equiv 0 \pmod{4}, {\text { or }} x\equiv \pm 1 \pmod{8}, \ y\equiv  2 \pmod{4}.
\end{align*}

\noindent  $qJ_{7,28}\overline{J}_{0,2}$ corresponds to the following weighted set of solutions to $8k+1=7x^2+2y^2$.  Here $x>0$ and $y$ is odd.  This is the excess of the number of inequivalent solutions with $x\equiv \pm 1 \pmod{8},$ over the number with $x\equiv \pm 3 \pmod{8}.$  For the right-hand side of (\ref{equation:NewKap2-0}), we consider the following weighted set of solutions to $8k+1=x^2+7y^2$.  Here $x>0$ is odd and $y\equiv 0 \pmod 4$.  This is the excess of the number of inequivalent solutions with $y\equiv 0 \pmod{8},$ over the number with $y\equiv 4 \pmod{8}.$

\vskip 0.1in
The case $p\equiv 7 \pmod 8$ of Theorem \ref{theorem:Mu-Thm3} follows from the following three identities:
\begin{align}
\overline{J}_{2,8}j(q^5;-q^7)+qJ_{1,4}J_{1,14}&=\overline{J}_{1,4}J_{10,28},\label{equation:Mu-Thm3B1}\\
\overline{J}_{2,8}j(q^3;-q^7)-J_{1,4}J_{5,14}&=q\overline{J}_{1,4}J_{6,28},\label{equation:Mu-Thm3B2}\\
\overline{J}_{2,8}j(q;-q^7)-J_{1,4}J_{3,14}&=q^2\overline{J}_{1,4}J_{2,28}.\label{equation:Mu-Thm3B3}
\end{align}

For the upcoming quadratic forms, we combine the inequivalent classes
\begin{equation*}
\{ (x,y), (x,-y)\} \text{ and } \{ (-x,-y), (-x,y)\},
\end{equation*}
into a single group and call it a solution set.

We discuss identity (\ref{equation:Mu-Thm3B1}).  For the left-hand side of (\ref{equation:Mu-Thm3B1}), we first consider $56k+23=x^2+14y^2$.  It is straightforward to show that $\overline{J}_{2,8}j(q^5,-q^7)$ is the generating function for the excess of the number of solution sets with
\begin{align*}
x\equiv \pm 3 \pmod {56}, \ y \text{ odd},  \text{ or } x\equiv \pm 11 \pmod {56}, \ y \text{ odd}
\end{align*}
over the number with
\begin{align*}
x\equiv \pm 17 \pmod {56}, \ y \text{ odd},  \text{ or } x\equiv \pm 25 \pmod {56}, \ y \text{ odd}.
\end{align*}
We now consider $56k+23=7x^2+2y^2$.  Here, $q{J}_{1,4}J_{1,14}$ is the generating function for the excess of the number of solution sets with
\begin{align*}
x\equiv \pm 1 \pmod {8}, \ y\equiv \pm 6 \pmod{28},  \text{ or } x\equiv \pm 3 \pmod {8}, \ y \equiv \pm 8 \pmod {28}
\end{align*}
over the number with
\begin{align*}
x\equiv \pm 1 \pmod {8}, \ y\equiv \pm 8 \pmod{28},  \text{ or } x\equiv \pm 3 \pmod {8}, \ y \equiv \pm 6 \pmod {28}.
\end{align*}
For the right-hand side of (\ref{equation:Mu-Thm3B1}), we consider $56k+23=x^2+7y^2$.  Here, $\overline{J}_{1,4}J_{10,28}$ is the generating function for the excess of the number of solution sets with $x\equiv \pm 4 \pmod {56}, \ y \text{ odd}$ over the number with $x\equiv \pm 24 \pmod {56}, \ y \text{ odd}.$

We discuss identity (\ref{equation:Mu-Thm3B2}).  For the left-hand side, we first consider $56k+71=x^2+14y^2$.  It is straightforward to show that $q^{-1}\overline{J}_{2,8}j(q^3,-q^7)$ is the generating function for the excess of the number of solution sets with
\begin{align*}
x\equiv \pm 1 \pmod {56}, \ y \text{ odd},  \text{ or } x\equiv \pm 15 \pmod {56}, \ y \text{ odd}
\end{align*}
over the number with
\begin{align*}
x\equiv \pm 13 \pmod {56}, \ y \text{ odd},  \text{ or } x\equiv \pm 27 \pmod {56}, \ y \text{ odd}.
\end{align*}
We now consider $56k+71=7x^2+2y^2$.  Here, $-q^{-1}{J}_{1,4}J_{5,14}$ is the generating function for the excess of the number of solution sets with
\begin{align*}
x\equiv \pm 1 \pmod {8}, \ y\equiv \pm 12 \pmod{28},  \text{ or } x\equiv \pm 3 \pmod {8}, \ y \equiv \pm 2 \pmod {28}
\end{align*}
over the number with
\begin{align*}
x\equiv \pm 1 \pmod {8}, \ y\equiv \pm 2 \pmod{28},  \text{ or } x\equiv \pm 3 \pmod {8}, \ y \equiv \pm 12 \pmod {28}.
\end{align*}
For the right-hand side of (\ref{equation:Mu-Thm3B2}), we consider $56k+71=x^2+7y^2$.  Here, $\overline{J}_{1,4}J_{6,28}$ is the generating function for the excess of the number of solution sets with $x\equiv \pm 8 \pmod {56}, \ y \text{ odd}$ over the number with $x\equiv \pm 20 \pmod {56}, \ y \text{ odd}.$

We discuss identity (\ref{equation:Mu-Thm3B3}).  For the left-hand side, we first consider $56k+151=x^2+14y^2$.  It is straightforward to show that $q^{-2}\overline{J}_{2,8}j(q,-q^7)$ is the generating function for the excess of the number of solution sets with
\begin{align*}
x\equiv \pm 5 \pmod {56}, \ y \text{ odd},  \text{ or } x\equiv \pm 19 \pmod {56}, \ y \text{ odd}
\end{align*}
over the number with
\begin{align*}
x\equiv \pm 9 \pmod {56}, \ y \text{ odd},  \text{ or } x\equiv \pm 23 \pmod {56}, \ y \text{ odd}.
\end{align*}
We now consider $56k+151=7x^2+2y^2$.  Here, $-q^{-2}{J}_{1,4}J_{3,14}$ is the generating function for the excess of the number of solution sets with
\begin{align*}
x\equiv \pm 1 \pmod {8}, \ y\equiv \pm 10 \pmod{28},  \text{ or } x\equiv \pm 3 \pmod {8}, \ y \equiv \pm 4 \pmod {28}
\end{align*}
over the number with
\begin{align*}
x\equiv \pm 1 \pmod {8}, \ y\equiv \pm 4 \pmod{28},  \text{ or } x\equiv \pm 3 \pmod {8}, \ y \equiv \pm 10 \pmod {28}.
\end{align*}
For the right-hand side of (\ref{equation:Mu-Thm3B3}), we consider $56k+151=x^2+7y^2$.  Here, $\overline{J}_{1,4}J_{2,28}$ is the generating function for the excess of the number of solution sets with $x\equiv \pm 12 \pmod {56}, \ y \text{ odd}$ over the number with $x\equiv \pm 16 \pmod {56}, \ y \text{ odd}.$
 
\begin{proof}[Proof of Theorem \ref{theorem:NewKap2}]  We first note that for a prime $p$
\begin{equation*}
p=x^2+7y^2 \iff p \equiv 1,9,11,15,23,25 \pmod{28}.
\end{equation*}
So if $p\equiv 1, 9,25,57,65,81 \pmod {112}$ is prime then are exactly two representations by $x^2+7y^2$ ($x>0$) with one obtained from the other by negating $y$.  By congruence considerations: 
\begin{itemize}
\item if $p=x^2+14y^2$ has a solution with positive weight then $p\equiv 1,65,81 \pmod{112}$,
\item if $p=x^2+14y^2$ has a solution with negative weight then $p\equiv 9,25,57 \pmod{112}$,
\item if $p=7x^2+2y^2$ has a solution with positive weight then $p\equiv 9,25,57 \pmod{112}$,
\item if $p=7x^2+2y^2$ has a solution with negative weight then $p\equiv 1,65,81 \pmod{112}$.
\end{itemize}

If $p\equiv 1,65,81 \pmod{112}$, the $p$'s unique representation of the form $x^2+7y^2$ ($x>0$, $y>0$) has $y\equiv 0 \pmod 8$ iff $p$ has a representation of the form $x^2+14y^2$ in which case $y$ even, i.e., iff $p$ has a representation of the form $x^2+56y^2.$

If $p\equiv 9,25,57 \pmod{112}$, the $p$'s unique representation of the form $x^2+7y^2$ ($x>0$, $y>0$) has $y\equiv 0 \pmod 8$ iff $p$ has a representation of the form $7x^2+2y^2$, i.e., iff $p$ does not have representation of the form $x^2+56y^2.$
\end{proof}

\begin{proof}[Proof of Theorem \ref{theorem:Mu-Thm3}]  We only prove $p\equiv 1 \pmod 8$; the case $p\equiv 7 \pmod 8$ is similar and will also be omitted.  We have two cases.  For the first case, we suppose $p\equiv 1,65,81 \pmod {112}$.  Thus $p=A^2+14B^2$ is solvable iff $p=M^2+7N^2$ has a solution with $M\equiv 1 \pmod 8$ and $N\equiv 0 \pmod 8$, i.e., iff $2p+M+N\equiv 3 \pmod 8.$  For the second case, we suppose $p\equiv 9,25,57 \pmod {112}$.  Thus $p=A^2+14B^2$ is solvable iff $p=M^2+7N^2$ has a solution with $M\equiv 5 \pmod 8$ and $N\equiv 4 \pmod 8$, i.e., iff $2p+M+N\equiv 3 \pmod 8.$  The argument for $p=7C^2+2D^2$ is similar.
\end{proof}
\section{Conclusion}  Putting Kaplansky-like theorems into the context of threefield identities enabled us to give new proofs of old theorems as well as to find theorems that were not on Brink's list.  Moreover, we see the shadow of the threefield identity from the new proof of Kaplansky's Theorem \ref{theorem:Kap-Thm} in results of Barrucand and Cohn \cite{BC} and Williams \cite{W}.  \section*{Acknowledgements}
We would like to thank Dean Hickerson for his help in finding identity (\ref{equation:master-id}) and its proof.  We would also like to thank Kenneth Williams and William Jagy for valuable feedback.


\begin{thebibliography}{999999}


\bibitem{Ap} M. P. Appell, {\em Sur les fonctions doublement p\'eriodiques de troisi\`eme esp\`ece}, Annales scientifiques de l'ENS, 3e s\'erie, t. I, p. \ 135, t. II, p.\ 9, t. III, p.\ 9, 1884-1886.

\bibitem{ADH} G. E. Andrews, F. J. Dyson, D. R. Hickerson, {\em Partitions and indefinite quadratic forms}, Invent. Math., {\bf 91} (1988), no. 3, pp. 391-407.

\bibitem{BC} P. Barrucand, H. Cohn, {\em Note on primes of type $x^2+32y^2$, class number and residuacity}, J. Reine Angew. Math., {\bf 238} (1969), pp. 67-70.

\bibitem{B1} D. Brink, {\em Five peculiar theorems on simultaneous representation of primes by quadratic forms}, Journal of Number Theory, {\bf 129} (2009), pp. 464-468.

\bibitem{B2} D. Brink, {\em Two theorems of Glaisher and Kaplansky}, Funct. Approx. Comment. Math., {\bf 41} (2009), pp. 163-165.

\bibitem{C} H. Cohen, {\em $q$-identities for Maass waveforms}, Invent. Math., {\bf 91} (1988), pp. 409-422.

\bibitem{G} C. F. Gauss, {\em Theorie der biquadratischen Reste, I,} in Arithmetische Untersuchungen, Chelsea reprint, 1969, pp. 511-533.

\bibitem{Gl} J. W. L. Glaisher, {\em On the expressions for the number of classes of a negative determinant, and on the numbers of positives in the octants of $P$}, Quart. J. Pure Appl. Math., {\bf 34} (1903), pp. 178-204.


\bibitem{He22} E. Hecke, {\em Uber einen Zusammenhang zwischen elliptischen Modulfunktionen und indefiniten quadratischen Formen}, Mathematische Werke, Vandenhoeck and Ruprecht, Goettingen, (1970), no. 22.

\bibitem{He23} E. Hecke, {\em Zur Theorie der elliptischen Modulfunktionen}, Mathematische Werke, Vandenhoeck and Ruprecht, Goettingen, (1970), no. 23.

\bibitem{HM} D. Hickerson, E. Mortenson, {\em Hecke-type double sums, Appell-Lerch sums, and mock theta functions (I)}, submitted.

\bibitem{K} I. Kaplansky, {\em The forms $x+32 y^2$ and $x+64y^2$}, Proc. AMS, {\bf 131} (2003), no. 7, pp. 2299-2300.

\bibitem{L1} M. Lerch, {\em Pozn\'amky k theorii funkc\'i elliptick\'ych}, Rozpravy \v Cesk\'e Akademie C\'isa\v re Franti\v ska Josefa pro v\v edy, slovesnost a um\v en\' i v praze, {\bf 24}, (1892), pp. 465-480.


\bibitem{Mu} J. B. Muskat, {\em On simultaneous representations of primes by binary quadratic forms}, Journal of Number Theory, {\bf 19} (1984), pp. 263-282.

\bibitem{MSW} J. B. Muskat, B. K. Spearman, K. S. Williams, {\em Predictive criteria for the representation of primes by binary quadratic forms}, Acta Arith. {\bf 70} (1995), pp. 215-278.


\bibitem{W} K. S. Williams, {\em Note on a result of Barrucand and Cohn}, J. Reine Angew. Math., {\bf 258} (1976), pp. 218-220.

\bibitem{W2} K. S. Williams, {\em Private communication.}

\end{thebibliography}
\end{document}